\documentclass[a4paper,10pt]{article}
\usepackage[utf8]{inputenc}
\usepackage[T1]{fontenc}
\usepackage[english]{babel}
\usepackage{amsmath,amssymb,amsfonts}
\usepackage{graphicx,color}
\usepackage{textcomp}
\usepackage{xcolor}
\usepackage{algorithm,algorithmic}
\usepackage{trfsigns}

\usepackage{listings}
\usepackage{mathrsfs}
\usepackage{stmaryrd}
\usepackage{bm}
\usepackage{tikz}
\usepackage{braket}
\usepackage{wasysym}
\usepackage{wrapfig}
\usepackage[framemethod=TikZ]{mdframed}
\usepackage{pstricks}
\usepackage{faktor}

\usepackage{amsthm}
\usepackage{graphicx}
\usepackage{hyperref}
\usepackage{subfig}
\usepackage[top=2.00cm,bottom=3.00cm,left=2.00cm,right=2.00cm]{geometry}

\DeclareMathOperator{\diag}{diag}

\newtheorem{theorem}{Theorem}

\newtheorem{example}{Example}

\title{On the Periodicity of Singular Vectors and the Holomorphic Block-Circulant SVD on the Unit Circumference}
\author{Giovanni Barbarino\\\footnotesize Mathematics and Operational Research Unit, Facult\'e polytechnique,\\ \footnotesize  Universit\'e de Mons, Belgium (giovanni.barbarino@umons.ac.be)}
\date{}

\begin{document}


\maketitle

\begin{abstract}
We investigate the singular value decomposition of a rectangular matrix that is analytic on the complex unit circumference, which occurs, e.g., with the matrix of transfer functions representing a broadband multiple-input multiple-output channel. Our analysis is based on the Puiseux series expansion of the eigenvalue decomposition of analytic para-Hermitian matrices on the complex unit circumference. We study the case in which the rectangular matrix does not admit a full analytic singular value factorization, either due to partly multiplexed systems or to sign ambiguity. We show how to find an SVD factorization in the ring of Puiseux series where each singular value and the associated singular vectors present the same period and multiplexing structure, and we prove that it is always possible to find an analytic pseudo-circulant factorization, meaning that any arbitrary arrangements of multiplexed systems can be converted into a parallel form. In particular, one can show that the sign ambiguity can be overcome by allowing non-real holomorphic singular values.
\end{abstract}

\section{Introduction}

Since its introduction \cite{1} the singular value decomposition (SVD) has been extensively used for a wide range of signal processing problems \cite{3,5,4}, from optimal precoder and equaliser design for multiple-input multiple-output (MIMO) channel matrices of complex gain factors \cite{6} to broadband problems whose channel matrix is determined by impulse responses \cite{8}. In the last case, though, the SVD ignores time or frequency correlations in the signals leading to loss of spectral coherence \cite{7}.



The focus has therefore shifted to the computation of the SVD for matrices with entries in functional rings and fields, in the hope to retain more information and regularity from the signals without having to process each time step or frequency separately. In this context, several algorithms for a polynomial SVD (PSVD) of possibly rectangular matrices have been recently developed, performing an approximate factorisation into Laurent-polynomial matrices. 
The use of PSVD algorithms extends from generic problems \cite{12,21}, finding application in various practical scenarios such as MIMO communications \cite{26,23,27,24,22,25}, the equalisation of filter bank-based multi-carrier systems \cite{28,29}, or broadband beamforming \cite{30}.


Despite the fact that the algorithms in \cite{15,12,16,13,14} are proven to converge to a diagonal matrix, it has not been until recently \cite{WPBP23} that questions about existence and uniqueness of such a decomposition have been raised to understand what these algorithms are actually converging to.
The purpose of this paper is to continue the investigation about the existence of the analytic SVD of a (possibly rectangular) matrix, and study how changing the constraints on the wanted factorization may lead to more regular decompositions.
The main contribution, in a signal processing context, is that any system can be brought by paraunitary
operations into a block-diagonal form containing pseudo-circulant matrices. This is quite profound:
it means that any arbitrary arrangement (concatenations, nesting, etc.) of multiplexed systems
can be converted and decoupled into a parallel form.

Consider a (possibly rectangular) matrix ${\bf A}(z)$ matrix that is holomorphic on the unit circumference $S^1$ . It can be equivalently represented by the Laurent time series with coefficients ${\bf A}[n]\laplace {\bf A}(z)$ or by the analytic and $2\pi$ periodic matrix on the real line ${\bf A}(\Omega):= {\bf A}(e^{j\Omega})$.
When $L$ is a positive integer, then with ${\bf B}(z^{1/L})$ we denote a matrix whose entries can be written as Laurent series in the variable $z^{1/L}$, called \textit{Puiseux series with index} $L$, that converges on an annulus around $S^1$, and can be equivalently represented by the analytic and $2\pi L$ periodic matrix on the real line ${\bf B}(\Omega):= {\bf B}(e^{j\Omega/L})$.
See \cite[Section 2.1]{BV22} for a more detailed discussion on the topic. From now on, any matrix, vector or scalar function depending on some complex variable $z$ or $z^{1/L}$ is intended to be analytic with respect to the relative variable at least on $S^1$.

In \cite{WPBP23} and \cite{BV22} it has been proven that any rectangular matrix ${\bf A}(z)$ admits an SVD 
\begin{equation}\label{eq:Puiseux_SVD}
    {\bf A}(z) = {\bf U}(z^{1/L}){\bf \Sigma}(z^{1/L}){\bf V}^P(z^{1/L})
\end{equation}
for some positive integer $L$, where ${\bf \Sigma}(z^{1/L})$ is a diagonal matrix and ${\bf U}(z^{1/L})$, ${\bf V}(z^{1/L})$ are para-unitary matrices, i.e. ${\bf U}(z) {\bf U}^P(z) = {\bf I}$ with the para-Hermitian operator $\{\cdot\}^P$ performing a time reversal and complex conjugation ${\bf U}^P(z) := [{\bf U}(1/\overline z)]^H$. 
Moreover, the diagonal entries of ${\bf \Sigma}(z^{1/L})$ are real for any $z\in S^1$, and thus their absolute values are exactly the singular values of ${\bf A}(z)$. The difference with a traditional SVD is that we allow for the singular values to be also negative, which is a necessary requisite in order to retain the analyticity of the factorization.

The parameter $L$ comes from the presence of multiplexed singular values and (an odd number of) zeros of the singular values $\sigma_i(z^{1/L})$ in ${\bf \Sigma}(z^{1/L})$. In fact, the singular values can be partitioned into orbits such that $k$-multiplexed singular values are in orbits of cardinality $k$, and  
\begin{align}
 \nonumber   &\sigma_{i+1}(\Omega) \equiv \sigma_i(\Omega+2\pi), \qquad  \forall i=1,\dots,k-1,\\
    &|\sigma_{1}(\Omega)| \equiv |\sigma_{k}(\Omega+2\pi)| \equiv |\sigma_{1}(\Omega+2\pi k)|, 
\end{align}
where $\sigma_i(\Omega) := \sigma_i(e^{j\Omega/L })$ and $|\sigma_{i}(\Omega)| \not \equiv |\sigma_{j}(\Omega)|$ for every $i\ne j$.
In case $\sigma_{1}(\Omega) \equiv \sigma_{k}(\Omega+2\pi)$, we say that the singular values $\sigma_i$ in this orbit are $k$-multiplexed and they are analytic in $z^{1/k}$. If instead $\sigma_{1}(\Omega) \equiv -\sigma_{k}(\Omega+2\pi)$  we say that the singular values are \textit{signed} $k$-multiplexed and in this case they are analytic in $z^{1/2k}$. 

As a way to get rid of the multiplexed index and return to classical holomorphic functions, it has been shown in  \cite{WPPC18} and \cite{BV22} that any square para-Hermitian matrix ${\bf A}(z) = {\bf A}^P(z)$ admits a decomposition
$${\bf A}(z) = {\bf U}(z){\bf C}(z){\bf U}^P(z),$$
where ${\bf U}(z)$ is para-unitary and ${\bf C}(z)$ is block diagonal and each block is pseudo-circulant, i.e. for each block of size $N$ there exist some functions $\phi_0(z),\dots, \phi_{N-1}(z)$ such that the block has the form	\[
	\begin{bmatrix}
	\phi_0(z) &z^{-1}\phi_{N-1}(z) &\dots & z^{-1}\phi_{1}(z)\\
	\phi_1(z) & \phi_0(z)& \ddots & \vdots\\
	\vdots& \ddots & \ddots& z^{-1}\phi_{N-1}(z)\\
	\phi_{N-1}(z)&\dots &\phi_1(z)  &\phi_0(z)
	\end{bmatrix}
	\]
and each block coincide with a different multiplexed orbit.
With respect to the above pseudo-circulant eigendecomposition for para-Hermitian matrices, whose blocks are determined by the multiplexing of the eigenvalues, when dealing with singular values we have also to be wary of possibly signed multiplexed orbits, that are less regular than the not signed ones. 

Here we show the existence of different "singular values" decompositions for any (possibly rectangular) matrix ${\bf A}(z)$, where we relax some constraints on the structure of the SVD in order to gain more regularity. 

First of all, we refine the decomposition in \cite{BV22,WPBP23} and find an analytic SVD 
$$\textbf{A}(e^{j\Omega}) = \textbf{U}(\Omega)
    {\bf \Sigma}(\Omega) 
    \textbf{V}(\Omega)^H,$$
where the diagonal matrix ${\bf \Sigma}(\Omega)$ may have negative entries and contains the singular values of ${\bf A}(\Omega)$ up to the sign, and for each nonzero singular value $\sigma(\Omega)$ that is (signed) $k$-multiplexed, the associated left and right singular vectors sports the same (signed) multiplexed behaviour. 

Then we allow for complex singular values in order to get rid of the signed multiplexed orbits, i.e. we find an analytic SVD
$$\textbf{A}(e^{j\Omega}) = \textbf{U}(\Omega)
    {\bf S}(\Omega) 
    \textbf{V}(\Omega)^H,$$
where  ${\bf S}(\Omega)$ is a diagonal and complex matrix, with diagonal entries $s_i(\Omega)$ whose absolute values correspond with the singular values of $\textbf{A}(e^{j\Omega})$. The functions $s_i(\Omega)$ are still multiplexed, but they are no more signed multiplexed.

Leveraging on the latest decomposition, we then prove that, even in the rectangular case, any ${\bf A}(z)$ admits a pseudo-circulant decomposition of the form 
$${\bf A}(z) = {\bf U}(z){\bf C}(z){\bf V}^P(z),$$
 where ${\bf C}(z)$ is block diagonal and each block is pseudo-circulant.

\section{Periodicity of Singular Vectors}

Let ${\bf A}$ be a (possibly rectangular) matrix with distinct non-zero singular values $\sigma_1,\dots,\sigma_p$, where each singular value has multiplicity $q_i$. In the SVD ${\bf A} = {\bf U}{\bf \Sigma}{\bf V}^H$, the matrix ${\bf \Sigma}$ is diagonal, real, nonnegative and contains the singular values $\sigma_i$ with the respective multiplicities on the diagonal, plus possibly some zeros. The factorization can be equivalently be written in dyadic notation as
\begin{equation}\label{eq:dyadic_svd}
   {\bf A} = \sum_{i=1}^p \sigma_i{\bf U}_i{\bf V}_i^H ,
\end{equation}
where ${\bf U}_i$ are the $q_i$ columns of ${\bf U}$ containing the left singular vectors relative to $\sigma_i$ and ${\bf V}_i$ are the $q_i$ columns of ${\bf V}$ containing the right singular vectors relative to $\sigma_i$. Even if the SVD is in general not unique, a known result is that each factor ${\bf A}_i:= {\bf U}_i{\bf V}_i^H$ in \eqref{eq:dyadic_svd} is uniquely determined by $\sigma_i$. 
Moreover, suppose the singular values in ${\bf \Sigma}$ are sorted so that repeated singular values are consecutive on its diagonal (not necessarily in some order). If now we take the two EVDs ${\bf A}{\bf A}^H={\bf Q}_1{\bf D}_1{\bf Q}_1^H$ and ${\bf A}^H{\bf A}={\bf Q}_2{\bf D}_2{\bf Q}_2^H$ where ${\bf D}_1$, ${\bf D}_2$ are real, diagonal, nonnegative and present the same nonzero eigenvalues $\lambda_i=\sigma_i^2$ on the same positions as the relative singular values appear in ${\bf \Sigma}$, then 
\begin{equation}\label{eq:SVD_from_EVD}
  {\bf A} = {\bf Q}_1{\bf \Sigma}{\bf \Psi}{\bf Q}_2^H,
\end{equation}
where ${\bf \Psi}$ is a square, block diagonal and unitary matrix, each block relative to a repeated singular value, with size equal to its multiplicity. 

When dealing with analytic matrices ${\bf A}(t)$ depending on on a real variable $t$, we already know that there exists a 'SVD' decomposition:
\begin{theorem}[Analytic SVD on a real interval, \cite{BBMN91}]\label{Th:Analytic_SVD_real}
For an $M\times N$ matrix ${\bf A}(t)$ that is analytic on some interval of $\mathbb R$, a decomposition 
\begin{equation}\label{eq:real_analytic_SVD}
    {\bf A}(t) = {\bf U}(t){\bf \Sigma}(t){\bf V}(t)^H
\end{equation}
exists with analytic unitary ${\bf U} (t)$ and ${\bf V}(t)$, and analytic, diagonal, real-valued ${\bf \Sigma}(t)$. 
\end{theorem}
The decomposition \eqref{eq:real_analytic_SVD} is not exactly a SVD since the diagonal elements of ${\bf \Sigma}(t)$ are permitted to be negative real numbers, but their absolute values coincide with the singular values of ${\bf A}(t)$. The negativity of the singular values, though, is an essential hypothesis to guarantee the analiticity of the decomposition, so from now on we call \eqref{eq:real_analytic_SVD} the analytic SVD of ${\bf A}(t)$. The necessity of this hypothesis is shown by the following example. 
\begin{example}
    Consider the $1\times 1$ matrix ${\bf A}(t) = 2\cos(t/2)e^{jt/2}$. A unitary $1\times 1$ matrix is just a unit norm complex number so in \eqref{eq:real_analytic_SVD} we can call $\gamma(t) := {\bf U}(t){\bf V}(t)^H$, $|\gamma(t)| = 1$, $\sigma(t):={\bf \Sigma}(t)\in\mathbb R$ and find that $2\cos(t)e^{jt/2} = \gamma(t)\sigma(t)$ from which $|2\cos(t/2)| = |\sigma(t)| = \pm \sigma(t)$ for every $t$. Since $\gamma(t)$ must be analytic, the only two possible factorization are $\sigma(t/2)\equiv 2\cos(t), \gamma(t)\equiv e^{jt/2}$ or  $\sigma(t)\equiv -2\cos(t), \gamma(t)\equiv -e^{jt/2}$ and in both cases $\sigma(t)$ is negative on some real multi-interval. \hfill \flushright{$\triangle$}
\end{example}
Consider now the matrix ${\bf A}(z)$ and its $2\pi$ periodic counterpart ${\bf A}(\Omega) ={\bf A}(e^{j\Omega})$. From \eqref{eq:real_analytic_SVD} there exists an analytic SVD ${\bf A}(\Omega) = {\bf U}(\Omega){\bf \Sigma}(\Omega){\bf V}(\Omega)^H$, but we do not have any prior information about the periodicity of the factors (or if they are periodic at all). From \cite{WPBP23} it has been shown that such a decomposition exists where all the factors are $2\pi L$ periodic for some positive integer $L$, so that one can reconstruct the decomposition ${\bf A}(z) = {\bf U}(z^{1/L}){\bf \Sigma}(z^{1/L}){\bf V}^P(z^{1/L})$ in Puiseux series of index $L$. In the following we study the periodicity of each singular value and vector, and we show that when the period of $\sigma_i(\Omega)$ is less than $2\pi L$ then its left and right singular vectors can present the same period of $\sigma_i(\Omega)$ or even half of it. 

\subsection{Multiplexing and Sign Changes}

The necessity of using the Puiseux series to find a SVD as in  \eqref{eq:Puiseux_SVD} is due to the possibility of the singular values being multiplexed or having a sign change. Recall that the sign problem is already present in the analytic SVD \eqref{eq:real_analytic_SVD} and the matrix in Example 1 already exhibits this issue:
\begin{example}
    Take ${\bf A}(z) = 1 + z$, whose associated matrix ${\bf A}(\Omega) = 1 + e^{j\Omega} = 2\cos(\Omega/2)e^{j\Omega/2}$ is the same as Example 1. Its only possible singular value, up to the sign, is thus $\sigma(z^{1/2}) = 2\cos(\Omega/2) = z^{j\Omega/2}+z^{-j\Omega/2}$.  \hfill \flushright{$\triangle$}  
\end{example}
Given a matrix ${\bf A}(z)$ and the analytic SVD ${\bf A}(\Omega) = {\bf U}(\Omega){\bf \Sigma}(\Omega){\bf V}(\Omega)^H$, we can always suppose that all the analytic singular values $\sigma_i(\Omega)$ are nonnegative in a small enough right interval of zero, i.e. $[0,\epsilon]$, up to a constant $\pm 1$ diagonal matrix that can be multiplied to ${\bf U}(\Omega)$ and ${\bf \Sigma}(\Omega)$. In particular, this choice assures us that if $\sigma_i(\Omega)\not\equiv\sigma_j(\Omega)$, then $|\sigma_i(\Omega)|\not\equiv|\sigma_j(\Omega)|$, because otherwise they would coincide on the interval $[0,\epsilon]$. Call then $\sigma_1(\Omega),\dots,\sigma_p(\Omega)$ the distinct nonzero singular values of ${\bf A}(\Omega)$ and, if it exists, its zero singular value $\sigma_0(\Omega)\equiv 0$. 

Let now $\Omega_0$ be a point such that all distinct singular values $\sigma_i(\Omega)$ have different absolute value when evaluated on $\Omega_0$. Notice that due to the analiticity of the singular values, $\Omega_0$ can be almost every real number, since the set of values $\Omega$ for which two distinct singular values take the same absolute value is discrete. 

Notice that the absolute values of the diagonal entries of ${\bf \Sigma}(\Omega)$ are always the classical singular values of ${\bf A}(\Omega)$, and since it is $2\pi$ periodic, then ${\bf A}(\Omega_0) = {\bf A}(\Omega_0+2\pi)$ so ${\bf \Sigma}(\Omega_0)$ and ${\bf \Sigma}(\Omega_0+2\pi)$ must contain the same diagonal entries up to signs and permutations.  In particular, this shows that 
\begin{itemize}
    \item the number of zero diagonal entries of ${\bf \Sigma}(\Omega_0)$ and ${\bf \Sigma}(\Omega_0+2\pi)$ is the same and equals the number of zero singular values of ${\bf A}(\Omega)$,
    \item if $|\sigma_i(\Omega_0+2\pi)| \ne |\sigma_i(\Omega_0)|$ then there exists $j\ne i$ such that $|\sigma_j(\Omega_0)| = |\sigma_i(\Omega_0+2\pi)|$ and they have the same multiplicity $q_i=q_j$,
    \item there exists a unique permutation $\tau$ of the indexes of the nonzero singular values such that $|\sigma_{\tau(i)}(\Omega_0)| = |\sigma_i(\Omega_0+2\pi)|$ for all $i=1,\dots,p$,
    \item since $\Omega_0$ can be almost every real number, and all $|\sigma_i(\Omega)|$ are at least continuous, then $|\sigma_{\tau(i)}(\Omega)| \equiv |\sigma_i(\Omega+2\pi)|$. 
\end{itemize}
Up to a renaming of the indexes, we can always suppose $\tau =(1,\dots,k_1)(k_1+1,\dots,k_1 + k_2)\dots(p-k_R+1\dots,p)$, meaning that the permutation $\tau$ partitions the $p$ distinct singular values into $R$ sets, or orbits, indexed by $\nu=1,\dots,R$, each of cardinality $k_\nu$, and cyclically shifts them. Moreover, up to changing the global sign of some of the singular values, we can drop almost all the absolute values in the above relation, except the one that links the last element of each cycle to the first. For example, looking at the first set, we find that 
\begin{align}\label{eq:first_set_sv_shift}
 \nonumber   &\sigma_{i+1}(\Omega) \equiv \sigma_i(\Omega+2\pi), \qquad  \forall i=1,\dots,k_1-1,\\
    &|\sigma_{1}(\Omega)| \equiv |\sigma_{k_1}(\Omega+2\pi)| \equiv |\sigma_{1}(\Omega+2\pi k_1)| ,
\end{align}
and since $\sigma_1(\Omega)$ is analytic and nonzero, then there exists $\kappa_1\in \{1,2\}$ such that 
$$
    \sigma_{1}(\Omega) \equiv (-1)^{\kappa_1-1} \sigma_{1}(\Omega+2\pi k_1)\equiv  \sigma_{1}(\Omega+2\pi k_1\kappa_1),
$$
and the same applies to all other singular values in the same set, i.e. 
\begin{equation}\label{eq:period_sv_with_sign}
    \sigma_{i}(\Omega) \equiv (-1)^{\kappa_1-1} \sigma_{i}(\Omega+2\pi k_1)\equiv  \sigma_{i}(\Omega+2\pi k_1\kappa_1)
\end{equation}
for every $i=1,\dots,k_1$. As a consequence all singular values of each $\nu$-orbit of $\tau$ has period $2\pi k_\nu\kappa_\nu$ where $k_\nu$ is the multiplexing index and at the same time the size of the $\nu$-set, whereas $\kappa_\nu\in\{1,2\}$ and it encodes the possible sign change. When $\kappa_\nu = 1$, we say that the singular values of the $\nu$-set are $k_\nu$-\textit{multiplexed}, and if additionally $\kappa_\nu=2$, we say that they are \textit{signed} $k_\nu$-\textit{multiplexed}. Notice that any (signed) $k_\nu$-multiplexed singular value is $2\pi k_\nu\kappa_\nu$ periodic and thus admits a Puiseux series in $z^{1/(k_\nu\kappa_\nu)}$.
For a more detailed exploration of multiplexed systems, see \cite{WPBP23}. 
\begin{example}\label{ex:1}
    Consider the matrix ${\bf A}(z)$
    $$\setlength\arraycolsep{1pt}
     \begin{pmatrix}
    2 & \sqrt 2 & \sqrt 2 & 0\\
    \sqrt 2 z^{-1} & 3 & -1 & \sqrt 2(1+z) \\
    \sqrt 2 z^{-1} & -1 & 3 & -\sqrt 2(1+z)    \\
   0 &\sqrt 2(1+z^{-1}) & -\sqrt 2(1+z^{-1}) & 4
\end{pmatrix}$$
whose real-variable equivalent ${\bf A}(\Omega)$ has analytic singular values $\sigma_1(\Omega) = 4\cos(\Omega/4)$, $\sigma_2(\Omega) = 4\sin(\Omega/4)$, $\sigma_3(\Omega) = 4+4\cos(\Omega/2)$ and $\sigma_4(\Omega)=4-4\cos(\Omega/2)$. Notice that 
    \begin{align*}
     &   |\sigma_1(\Omega+2\pi)| = 4|\cos(\Omega/4 + \pi/2)| = |\sigma_2(\Omega)|,\\
     &   |\sigma_2(\Omega+2\pi)| = 4|\sin(\Omega/4 + \pi/2)| = |\sigma_1(\Omega)|,\\
     &   |\sigma_3(\Omega+2\pi)| = |4+4\cos(\Omega/2 + \pi)| =|\sigma_4(\Omega)|,\\
     &   |\sigma_4(\Omega+2\pi)| = |4-4\cos(\Omega/2 + \pi)| =|\sigma_3(\Omega)|,
    \end{align*}
    so that the resulting permutation is $\tau=(1,2)(3,4)$ and $k_1=k_2=2$. Regarding the sign change, it is enough to test it on $\sigma_1$ and  $\sigma_3$:
    \begin{align*}
     &  \sigma_1(\Omega + 4\pi) = 4\cos(\Omega/4 + \pi) = -\sigma_1(\Omega), \\
     &   \sigma_3(\Omega+4\pi) = 4+4\cos(\Omega/2 + 2\pi) =\sigma_3(\Omega),
    \end{align*}
showing that $\kappa_1 = 2$, $\kappa_2=1$. In fact $\sigma_1$ and $\sigma_2$ have period $2\pi k_1\kappa_1 = 8\pi$ and are signed $2$-multiplexed, whereas $\sigma_3$ and $\sigma_4$ have period $2\pi k_2\kappa_2 = 4\pi$ and are $2$-multiplexed.  \hfill \flushright{$\triangle$}
\end{example}

\subsection{Periodicity of the Singular Vectors}

Let us now focus on the periodicity of the singular vectors. In the previous section we showed that each singular value has periodicity $2\pi k_\nu\kappa_\nu$, where $k_\nu$ is the size of the multiplexed set and $\kappa_\nu$ takes care of the possible sign change. 

From \cite{WPBP23} we know that any ${\bf A}(z)$ admits a SVD in Puiseux series with index $L=lcm\{k_1\kappa_1,\dots, k_R\kappa_R\}$, where $R$ is the number of the orbits of the associated permutation.
The SVD can be reformulated with the associated real-variable matrices as ${\bf A}(\Omega) = {\bf U}(\Omega){\bf \Sigma}(\Omega){\bf V}(\Omega)^H$ where all the factors have period $2\pi L$, but in general the elements of ${\bf \Sigma}(\Omega)$ have smaller periods $2\pi k_\nu\kappa_\nu$. 

Here we prove that one can always take ${\bf U}(\Omega)$ such that the left singular vectors associated to (signed) $k_\nu$-multiplexed singular values are in turn $k_\nu$-multiplexed, i.e. if $(r,r+1,\dots,r+k_\nu-1)$ is the orbit of the singular value $\sigma_i(\Omega)$ and ${\bf U}_i(\Omega)$ are the left singular vectors in ${\bf U}(\Omega)$ associated to $\sigma_i(\Omega)$, then 
\begin{itemize}
    \item ${\bf U}_s(\Omega+2\pi) = {\bf U}_{s+1}(\Omega)$ for $s=r,\dots,r+k_\nu-1$,  
    \item ${\bf U}_{r+k_\nu -1}(\Omega+2\pi) = {\bf U}_{r}(\Omega)$,
\end{itemize}
and in particular ${\bf U}_i(\Omega)$ is $2\pi k_\nu$ periodic. 

At the same time, one can always take ${\bf V}(\Omega)$ such that right singular vectors associated to (signed) $k_\nu$-multiplexed singular values are in turn (signed) $k_\nu$-multiplexed, i.e. if ${\bf V}_i(\Omega)$ are the right singular vectors in ${\bf V}(\Omega)$ associated to $\sigma_i(\Omega)$, then 
\begin{itemize}
    \item ${\bf V}_s(\Omega+2\pi) = {\bf V}_{s+1}(\Omega)$ for $s=r,\dots,r+k_\nu-1$,
    \item ${\bf V}_{r+k_\nu-1}(\Omega+2\pi ) \equiv (-1)^{\kappa_1-1}{\bf V}_{r}(\Omega)$,
\end{itemize}
and in particular ${\bf V}_i(\Omega)$ is $2\pi k_i\kappa_i$ periodic.

The following proof coincide in part with Theorem 2 in \cite{WPBP23} and Theorem 3.14 in \cite{BV22}. We make use of Proposition 3.9 of \cite{BV22}, where it is proved that for any para-Hermitian matrix ${\bf A}(z)$ there exists an EVD ${\bf A}(\Omega) = {\bf W}(\Omega){\bf D}(\Omega){\bf W}(\Omega)^H$ where the eigenvectors in ${\bf W}(\Omega)$ have the same period as the respective eigenvalues on the diagonal of ${\bf D}(\Omega)$, and from the proof of Theorem 3.14 in the same document, we can also suppose that the eigenvectors ${\bf W}_i(\Omega)$ relative to an eigenvalue $\lambda_i(\Omega)$ are $k_\nu$-multiplexed when the eigenvalue itself is $k_\nu$-multiplexed. 

\begin{theorem}\label{th:multiplexed_singular_vectors_SVD}
For an analytic matrix $\textbf{A}(z)$, there exists an analytic SVD on the unit circumference 
$$\textbf{A}(e^{j\Omega}) = \textbf{U}(\Omega)
    {\bf \Sigma}(\Omega) 
    \textbf{V}(\Omega)^H,$$
where the diagonal matrix ${\bf \Sigma}(\Omega)$ may have negative entries and contains the singular values of ${\bf A}(\Omega)$ up to the sign. Moreover, each nonzero singular value $\sigma_i(\Omega)$ in $ {\bf \Sigma}(\Omega) $ is (signed) $k_\nu$-multiplexed and
\begin{itemize}
    \item    the associated left singular vectors $\textbf{U}_i(\Omega)$ are $k_\nu$-multiplexed, and in particular $2\pi k_\nu$ periodic, 
    \item    the associated right singular vectors $\textbf{V}_i(\Omega)$ are (signed) $k_\nu$-multiplexed, and in particular $2\pi k_\nu\kappa_\nu$ periodic.
\end{itemize}
\end{theorem}

\begin{proof}
From Theorem \ref{Th:Analytic_SVD_real}, the SVD ${\bf A}(\Omega) =\widetilde {\bf U}(\Omega){\bf \Sigma}(\Omega)\widetilde{\bf V}(\Omega)^H$ presents nonzero singular values that are $2\pi k_\nu \kappa_\nu$ periodic depending on the orbits of the permutation. Without loss of generality, suppose the permutation is $\tau=(1,\dots,k_1)(k_1+1,\dots)\dots$ and notice that from \eqref{eq:period_sv_with_sign},
$$\sigma_{i}(\Omega)^2 \equiv  \sigma_{i}(\Omega+2\pi k_\nu)^2 \quad \forall i,$$
i.e., the square of all singular values are $2\pi k_\nu$ periodic. Both ${\bf R}_1(z):= {\bf A}(z){\bf A}^P(z)$ and ${\bf R}_2(z):= {\bf A}^P(z){\bf A}(z)$ are positive semi-definite para-Hermitian matrices and have the same nonzero analytic eigenvalues $\lambda_i(\Omega)$ that are the square of the nonzero singular values in ${\bf \Sigma}(\Omega)$.
We can always suppose, up to a sign change, that ${\bf \Sigma}(\Omega)\ge 0$ on a small interval $\Omega\in [0,\varepsilon]$, so that if $\sigma_i(\Omega) = -\sigma_j(\Omega)$ for all $\Omega$, then necessarily $\sigma_i \equiv \sigma_j\equiv 0$. As a consequence, there's a unique way to associate the nonzero singular values $\sigma_i(\Omega)$ to the nonzero eigenvalues $\lambda_i(\Omega)=\sigma_i(\Omega)^2$ and in particular they have the same multiplicity, and if $\sigma_i(\Omega)$ is (signed) $k_\nu$-multiplexed, then the eigenvalue $\lambda_i(\Omega)$ is $k_\nu$-multiplexed with the same orbit.
One can thus write down the EVD of the para-Hermitian matrices as
\begin{align*}
    &{\bf R}_1(\Omega) = {\bf Q}_1(\Omega){\bf D}_1(\Omega){\bf Q}_1(\Omega)^H,\\
    &{\bf R}_2(\Omega) = {\bf Q}_2(\Omega){\bf D}_2(\Omega){\bf Q}_2(\Omega)^H,
\end{align*}
where the eigenvectors in ${\bf Q}_1(\Omega)$ and ${\bf Q}_2(\Omega)$ relative to the eigenvalue $\lambda_i(\Omega)\not\equiv 0$ have the same period $2\pi k_\nu$ and are $k_\nu$-multiplexed when $\lambda_i(\Omega)$ and $\sigma_i(\Omega)$ are $k_\nu$-multiplexed, and ${\bf D}_1(\Omega)$, ${\bf D}_2(\Omega)$ present in order the same nonzero eigenvalues on their diagonals. From \eqref{eq:SVD_from_EVD}, we can then formulate a SVD of ${\bf A}(\Omega)$ as ${\bf A}(\Omega) = {\bf Q}_1(\Omega){\bf \Sigma}(\Omega){\bf \Psi}(\Omega){\bf Q}_2^H(\Omega)$ only for the $\Omega$ where ${\bf \Sigma}(\Omega)\ge 0$, but if some singular value is negative, we can always multiply a diagonal $\pm 1$ matrix to ${\bf \Psi}(\Omega)$ that keeps it unitary and with the same blocks on the diagonal. Moreover, we can always suppose that the elements on the diagonal of ${\bf \Psi}(\Omega)$ relative to identically zero singular values are just $1$.   

Let  ${\bf Q}_1^{(i)}(\Omega)$ and $[{\bf Q}_2(\Omega){\bf \Psi}(\Omega)^H]^{(i)} = {\bf Q}^{(i)}_2(\Omega){\bf \Psi}_i(\Omega)^H$  be the left and right singular vectors relative to tho nonzero singular value $\sigma_i(\Omega)$. Recall that up to a negligible set, for all $\Omega\in \mathbb R$ the values $|\sigma_i(\Omega)|$ are distinct, so from now on we always suppose that it holds. As a consequence, by the uniqueness of \eqref{eq:dyadic_svd} and the analytic SVD ${\bf A}(\Omega) = \widetilde{\bf U}(\Omega){\bf \Sigma}(\Omega)\widetilde{\bf V}(\Omega)^H$ we have
\begin{align*}
  &{\bf Q}_1^{(i)}(\Omega){\bf \Psi}_i(\Omega) {\bf Q}^{(i)}_2(\Omega)^H =\widetilde{\bf U}_i(\Omega)\widetilde{\bf V}_i(\Omega)^H\\
  \implies &{\bf \Psi}_i(\Omega) = {\bf Q}_1^{(i)}(\Omega)^H
  \widetilde{\bf U}_i(\Omega)\widetilde{\bf V}_i(\Omega)^H
  {\bf Q}^{(i)}_2(\Omega),
\end{align*}
proving that each block of ${\bf \Psi}(\Omega)$, and thus the whole ${\bf \Psi}(\Omega)$, is also analytic. 

We can thus consider the two unitary analytic matrices ${\bf U}(\Omega):={\bf Q}_1(\Omega)$ and ${\bf V}(\Omega):={\bf Q}_2(\Omega){\bf \Psi}(\Omega)^H$ and form the analytic SVD ${\bf A}(\Omega) = {\bf U}(\Omega){\bf \Sigma}(\Omega){\bf V}(\Omega)^H$ where ${\bf U}(\Omega)$ already satisfies the thesis. Recall that the permutation is $\tau=(1,\dots,k_1)(k_1+1,\dots)\dots$, so we can focus on the first set of multiplexed singular values, for which we can always suppose
\begin{align*}
    &\sigma_{i+1}(\Omega) \equiv \sigma_i(\Omega+2\pi), &  \forall i=1,\dots,k_1-1,\\
    &\sigma_{i}(\Omega) \equiv (-1)^{\kappa_1-1} \sigma_{i}(\Omega+2\pi k_1), &  \forall i=1,\dots,k_1.
\end{align*}
 As a consequence, from the uniqueness of the decomposition \eqref{eq:dyadic_svd}, the $2\pi$ periodicity of ${\bf A}(\Omega)$, and the fact that ${\bf U}_i(\Omega)$ is $k_\nu$-multiplexed we find that
\begin{align*}
    {\bf U}_{i+1}(\Omega ){\bf V}_i(\Omega+2\pi)^H &= 
{\bf U}_i(\Omega+2\pi ){\bf V}_i(\Omega+2\pi)^H \\&= {\bf U}_{i+1}(\Omega){\bf V}_{i+1}(\Omega)^H\\
\implies &\quad  {\bf V}_i(\Omega+2\pi) ={\bf V}_{i+1}(\Omega)
\end{align*}
for any $i=1,\dots,k_1-1$ and 
\begin{align*}
    {\bf U}_{i}(\Omega){\bf V}_{i}(\Omega+2\pi k_1)^H &= 
{\bf U}_{i}(\Omega+2\pi k_1){\bf V}_{i}(\Omega+2\pi k_1)^H \\&= (-1)^{\kappa_1-1} {\bf U}_{i}(\Omega){\bf V}_{i}(\Omega)^H\\
\implies &\quad {\bf V}_{i}(\Omega+2\pi k_1) = (-1)^{\kappa_1-1}{\bf V}_i(\Omega)
\end{align*}
for any $i=1,\dots,k_1$. Repeating the same steps for all orbits, the thesis is proved.
\end{proof}

\begin{example}
Let ${\bf A}(z)$ be the same matrix as in Example \ref{ex:1} and consider the unitary matrices
\begin{align*}
    {\bf U}(\Omega) &:= \setlength\arraycolsep{2pt}\frac 12
\begin{pmatrix}
    \sqrt 2& \sqrt 2 &0 &0\\
    e^{-j\Omega/2}&-e^{-j\Omega/2}&1&1\\
    e^{-j\Omega/2}&-e^{-j\Omega/2}&-1&-1\\
    0&0&\sqrt 2 e^{-j\Omega/2}&-\sqrt 2 e^{-j\Omega/2}
\end{pmatrix}
,\\
{\bf V}(\Omega) &:= \setlength\arraycolsep{.5pt} \frac 12
\begin{pmatrix}
    \sqrt 2 e^{j\Omega/4} & j\sqrt 2e^{j\Omega/4}& 0&0\\
    e^{-j\Omega/4} & -je^{-j\Omega/4} & 1 & 1 \\
   e^{-j\Omega/4} & -je^{-j\Omega/4} & -1 & -1 \\
    0&0 & \sqrt 2e^{-j\Omega/2} & -\sqrt 2e^{-j\Omega/2}
\end{pmatrix}.
\end{align*}
Then ${\bf A}(\Omega) = {\bf U}(\Omega){\bf \Sigma}(\Omega){\bf V}(\Omega)^H$ and the singular values on the diagonal of ${\bf \Sigma}(\Omega)$ are in order $\sigma_1(\Omega) = 4\cos(\Omega/4)$, $\sigma_2(\Omega) = -4\sin(\Omega/4)$, $\sigma_3(\Omega) = 4+4\cos(\Omega/2)$ and $\sigma_4(\Omega)=4-4\cos(\Omega/2)$. In this case $\sigma_1(\Omega)$ and $\sigma_2(\Omega)$ are signed $2$-multiplexed, and in fact ${\bf U}_1(\Omega)$, ${\bf U}_2(\Omega)$ are $2$-multiplexed since
$${\bf U}_1(\Omega+4\pi) = {\bf U}_2(\Omega+2\pi) = {\bf U}_1(\Omega)$$
and ${\bf V}_1(\Omega)$, ${\bf V}_2(\Omega)$ are signed $2$-multiplexed since
$${\bf V}_1(\Omega+4\pi) = {\bf V}_2(\Omega+2\pi) = -{\bf V}_1(\Omega).$$
Similarly, $\sigma_3(\Omega)$ and $\sigma_4(\Omega)$ are $2$-multiplexed, and in fact one can check that ${\bf U}_3(\Omega)$, ${\bf U}_4(\Omega)$, ${\bf V}_3(\Omega)$, ${\bf V}_4(\Omega)$ are all $2$-multiplexed. \hfill \flushright{$\triangle$}
\end{example}

\section{Diagonal Complex Decomposition}

 {
}

Here we discuss how to replace the signed multiplexed singular values and vectors with non-signed ones, in fact raising the regularity of the involved quantities by reducing their periods by a factor 2. In exchange, we have to admit that the singular values may be complex-valued. 
In fact, if $\sigma_1(\Omega),\dots,\sigma_k(\Omega)$ and the relative right singular vectors in ${\bf V}(\Omega)$ in the decomposition of Theorem \ref{th:multiplexed_singular_vectors_SVD} are signed $k$-multiplexed, then $s_i(\Omega) := \sigma_i(\Omega)e^{j\Omega/(2k)}\omega_{2k}^{i-1}$ and $\widetilde {\bf V}_i(\Omega) := {\bf V}_i(\Omega)e^{j\Omega/(2k)}\omega_{2k}^{i-1}$ are just $k$-multiplexed,  where $\omega_{p}:= e^{2\pi j/p}$ and $|s_i(\Omega)| = |\sigma_i(\Omega)|$ are the singular values of ${\bf A}(\Omega)$. As a consequence, it is sufficient to multiply ${\bf \Sigma}(\Omega)$ and ${\bf V}(\Omega)$ by an opportune unitary matrix to get rid of the sign change in the multiplexed singular values and vectors.

\begin{theorem}\label{th:complex_SVD}
For an analytic matrix $\textbf{A}(z)$, there exists an analytic complex diagonalization on the unit circumference 
$$\textbf{A}(e^{j\Omega}) = \textbf{U}(\Omega)
    {\bf S}(\Omega) 
    \textbf{V}(\Omega)^H,$$
where the diagonal matrix ${\bf S}(\Omega)$ may have complex entries whose absolute values are the singular values of ${\bf A}(\Omega)$ and ${\bf U}(\Omega)$, ${\bf V}(\Omega)$ are paraunitary. Moreover, if the nonzero singular value $\sigma_i(\Omega)$ is (signed) $k_\nu$-multiplexed, then the relative nonzero entry $s_i(\Omega)$, the  associated left singular vectors $\textbf{U}_i(\Omega)$ and the associated right singular vectors $\textbf{V}_i(\Omega)$ are all $k_\nu$-multiplexed, and in particular $2\pi k_\nu$ periodic.
\end{theorem}

\begin{proof}
Thanks to Theorem \ref{th:multiplexed_singular_vectors_SVD}, there always exists an analytic decomposition
$${\bf A}(\Omega) = {\bf U}(\Omega)
{\bf \Sigma}(\Omega)
\widetilde {\bf V}(\Omega)^H,$$
where for each (signed) $k_\nu$-multiplexed nonzero singular value $\sigma_i(\Omega)$ in $ {\bf \Sigma}(\Omega)$, the associated left singular vectors $  {\bf U}_i(\Omega)$ are $k_\nu$-multiplexed, and the associated right singular vectors $ \widetilde {\bf V}_i(\Omega)$ are (signed) $k_\nu$-multiplexed.

Let us now focus on the first set of singular values $\sigma_i(\Omega)$ with $1\le i\le k_1$ that are (signed) $k_1$-multiplexed and define $s_i(\Omega) := \sigma_i(\Omega)e^{j\Omega/(k_1\kappa_1)}
\omega_{k_1\kappa_1}^{i-1}$. All $s_i(\Omega)$ are $k_1$-multiplexed since 
\begin{align*}
   s_i(\Omega +2\pi) =&\,\sigma_i(\Omega+2\pi)e^{j\Omega/(k_1\kappa_1)}\omega_{k_1\kappa_1}\omega_{k_1\kappa_1}^{i-1}\\
   = &\,\sigma_{i+1}(\Omega)e^{j\Omega/(k_1\kappa_1)}\omega_{k_1\kappa_1}^{i}
=s_{i+1}(\Omega)\end{align*}
for $i=1,\dots,k_1-1$ and 
\begin{align*}
    s_i(\Omega +2\pi k_1)  =&\,\sigma_i(\Omega +2\pi k_1)e^{j\Omega/(k_1\kappa_1)}e^{j2\pi/\kappa_1}\omega_{k_1\kappa_1}^{i-1}\\
    =&\,\sigma_i(\Omega)e^{j\Omega/(k_1\kappa_1)}\omega_{k_1\kappa_1}^{i-1} =s_i(\Omega)
\end{align*}
for $i=1,\dots,k_1$. Using the same reasoning, one can prove also that ${\bf V}_i(\Omega):= \widetilde{\bf V}_i(\Omega)e^{j\Omega/(k_1\kappa_1)}\omega_{k_1\kappa_1}^{i-1}$ are $k_1$-multiplexed. 
Repeating the same reasoning for all the (copies of all) the orbits of $\tau$, we find a decomposition of the form  $${\bf A}(\Omega) = \sum_{i}s_i(\Omega){\textbf{U}}_{i}(\Omega)
   {\textbf{V}}_{i}(\Omega)^H = {\textbf{U}}(\Omega){\bf S}(\Omega)
   {\textbf{V}}(\Omega)^H$$
 with diagonal complex matrix ${\bf S}(\Omega)$.
\end{proof}

\begin{example}
\label{ex:2}Let ${\bf A}(z)$ be
$$
{\bf A}(z) := \frac 12
\begin{pmatrix}
    -\sqrt 2j(z^{-2} +z) & -j(z+z^{-1}) & -j(z+z^{-1})\\
    \sqrt 2(z - z^{-2} ) & z-z^{-1} & z-z^{-1}
\end{pmatrix},$$
that has left and right singular vectors
\begin{align*}
    {\bf U}(\Omega) &:=-\frac 12
\begin{pmatrix}
    j(e^{-j2\Omega} +e^{j\Omega/2})   & j(e^{-j2\Omega}- e^{j\Omega/2})\\
  e^{-j2\Omega} - e^{j\Omega/2}   &e^{-j2\Omega} +e^{j\Omega/2}
\end{pmatrix}
,\\
{\bf V}(\Omega) &:=  \frac 12
\begin{pmatrix}
    \sqrt 2 e^{-j\Omega/4} & -\sqrt 2 j e^{-j\Omega/4}& 0\\
   e^{-j3\Omega/4} & je^{-j3\Omega/4} & \sqrt 2  \\
   e^{-j3\Omega/4} & j e^{-j3\Omega/4} & -\sqrt 2
\end{pmatrix}
\end{align*}
and SVD ${\bf A}(\Omega) = {\bf U}(\Omega){\bf \Sigma}(\Omega){\bf V}(\Omega)^H$ with $\sigma_1(\Omega) = 2\cos(\Omega/4)$, $\sigma_2(\Omega) = -2\sin(\Omega/4)$. In this case $\sigma_1(\Omega)$ and $\sigma_2(\Omega)$ are signed $2$-multiplexed, so $k_1=\kappa_1=2$. Moreover, ${\bf U}_1(\Omega)$, ${\bf U}_2(\Omega)$ are $2$-multiplexed and ${\bf V}_1(\Omega)$, ${\bf V}_2(\Omega)$ are signed $2$-multiplexed. We eliminate the sign ambiguity in transforming the singular values into $s_i(\Omega)$ by multiplication with $e^{j\Omega/(k_1\kappa_1)}\omega_{k_1\kappa_1}^{i-1} = e^{j\Omega/4}
j^{i-1}$, thus
\[
s_1(\Omega) = \sigma_1(\Omega)e^{j\Omega/4} =2\cos(\Omega/4)e^{j\Omega/4} ,\quad  s_2(\Omega) = \sigma_2(\Omega)e^{j\Omega/4}j = -2\sin(\Omega/4)e^{j\Omega/4}j.
\]
Notice that they are 2-multiplexed without the sign ambiguity since
\begin{align*}
    s_1(\Omega+2\pi) &= 2\cos(\Omega/4+\pi/2)e^{j\Omega/4 +j\pi/2}
    = -2\sin(\Omega/4)e^{j\Omega/4}j = s_2(\Omega),\\
    s_2(\Omega+2\pi) &= -2\sin(\Omega/4 + \pi/2)e^{j\Omega/4 + j\pi/2}j
    = 2\cos(\Omega/4)e^{j\Omega/4 } =s_1(\Omega).
\end{align*}
If ${\bf S}(\Omega)$ is a diagonal matrix with $s_i(\Omega)$ as diagonal elements, then
$${\bf A}(\Omega) = {\bf U}(\Omega){\bf S}(\Omega)
\begin{pmatrix}
    e^{j\Omega/4} & &  \\
    & e^{j\Omega/4}j & \\
    & & 1
\end{pmatrix}^H
{\bf V}(\Omega)^H$$
where the first two columns of the unitary matrix
\[
{\bf V}(\Omega)\begin{pmatrix}
    e^{j\Omega/4} & &  \\
    & e^{j\Omega/4}j & \\
    & & 1
\end{pmatrix} = 
\frac 12
\begin{pmatrix}
    \sqrt 2 & \sqrt 2 & 0\\
   e^{-j\Omega/2} & -e^{-j\Omega/2} & \sqrt 2  \\
   e^{-j\Omega/2} & - e^{-j\Omega/2} & -\sqrt 2
\end{pmatrix}
\]
are now $2$-multiplexed with no sign ambiguity.\hfill \flushright{$\triangle$}
\end{example}

\section{Diagonal Pseudo-Circulant Decomposition}

The multiplexed singular values are tightly linked to the eigenvalues of certain structured matrices, called \textit{pseudo-circulant}, that are $2\pi$ periodic $N\times N$ normal matrices ${\bf C}(\Omega)$ for which there exist analytic $2\pi$ periodic functions $\phi_0(\theta),\dots, \phi_{N-1}(\theta)$ such that ${\bf C}(\Omega)$ is in the form	\[ 
	\begin{bmatrix}
	\phi_0(\Omega) &e^{-j\Omega}\phi_{N-1}(\Omega) &\dots & e^{-j\Omega}\phi_{1}(\Omega)\\
	\phi_1(\Omega) & \phi_0(\Omega)& \ddots & \vdots\\
	\vdots& \ddots & \ddots& e^{-j\Omega}\phi_{N-1}(\Omega)\\
	\phi_{N-1}(\Omega)&\dots &\phi_1(\Omega)  &\phi_0(\Omega)
	\end{bmatrix}
	\]
One can prove that the eigenvalues $\lambda_i(\Omega)$ are all $N$-multiplexed and all pseudo-circulant matrices are diagonalized by the same unitary transformation
\begin{equation}\label{eq:WN}
    {\bf W}_N(\Omega) := \diag(e^{j\Omega k/N})_{k=0,\dots,N-1}{\bf F}_N 
= {\bf D}_N(\Omega){\bf F}_N,
\end{equation}
 where ${\bf F}_N$ is the $N\times N$ unitary Fourier matrix. The eigenvalues can be expressed in term of the $\phi_k(\Omega)$ as
$$\lambda_i(\Omega) = \sum_{k=0}^{N-1} e^{-j\Omega k/N}\phi_k(\Omega)\omega_N^{-ki}, $$
 where $\omega_N:= e^{j 2\pi/N}$.
 
Moreover, any set $\{\lambda_1(\Omega),\dots,\lambda_N(\Omega)\}$ of analytic $N$-multiplexed functions is the set of eigenvalues for some pseudo-circulant matrix, i.e.
$${\bf W}_N(\Omega)\diag(\lambda_k(\Omega))_{k=1,\dots,N} {\bf W}_N(\Omega)^H$$
is always pseudo-circulant. The functions $\phi_k(\Omega)$ can be expressed in terms of the eigenvalues as
$$ \phi_{k}(\Omega)
 = 
 \frac 1N
 e^{j\Omega k/N}
 \sum_{i=0}^{N-1}	
 \lambda_i(\Omega)
 \omega_N^{ki}.$$
An important technical result in \cite{BV22} proves that the columns of ${\bf W}_N(\Omega)$ are $N$-multiplexed, i.e.
\begin{equation}\label{eq:WN_multiplexed}
   {\bf W}_N(\Omega+2\pi) = {\bf W}_N(\Omega)P_N,
\end{equation}
where $P_N$ is the constant permutation matrix
\begin{equation}\label{eq:PN}
    \begin{bmatrix}
	& & & 1\\
	1 & & & \\
	& \ddots & &\\
	& & 1 &
	\end{bmatrix}.
\end{equation}

The proof of the main result showing the existence of the holomorphic pseudo-circulant block decomposition follows the same arguments of Theorem 3.14 in \cite{BV22}. The main idea is to group the multiplexed singular values in the diagonal complex decomposition from Theorem \ref{th:complex_SVD} and apply  the base change ${\bf W}(\Omega)$ to make them block pseudo-circulant. Notice that this wasn't possible to achieve in presence of signed multiplexed singular values. 
The unitary matrices undergo the same base change, but since they are also multiplexed thanks to Theorem \ref{th:complex_SVD},  they also become $2\pi$-periodic. 

\begin{theorem}\label{th:pseudocirculant_SVD}
For an analytic matrix $\textbf{A}(z)$, there exists an analytic decomposition 
$${\bf A}(z) = {\bf U}(z) {\bf C}(z) {\bf V}(z)^P$$
with para-unitary matrices ${\bf U}(z)$, ${\bf V}(z)$ and pseudo-circulant block-diagonal matrix ${\bf C}(z)$, where each block has size $k_\nu$ corresponding to a set of (signed) $k_\nu$ singular values.  
\end{theorem}

\begin{proof}
Thanks to Theorem \ref{th:complex_SVD}, there always exists an analytic decomposition
$${\bf A}(\Omega) = \widetilde{\bf U}(\Omega)
{\bf S}(\Omega)
\overline {\bf V}(\Omega)^H,$$
where for each $k_\nu$-multiplexed nonzero singular value $s_i(\Omega)$ in $ {\bf S}(\Omega)$, the associated left singular vectors $ \widetilde {\bf U}_i(\Omega)$ and the associated right singular vectors $ \overline {\bf V}_i(\Omega)$ are all $k_\nu$-multiplexed. More importantly, there are no singular values nor singular vectors that are signed multiplexed.

If $q$ is the multiplicity of the singular value $s_1(\Omega)$, then there are $q$ identical copy of the first set of singular values, and $q$ is also the number of columns in $\overline{\bf V}_i(\Omega)$ and $\widetilde {\bf U}_i(\Omega)$ for  $i=1,\dots,k_1$. When reordering the repeated singular values into forming the $q$ identical multiplexed sets, we can always suppose that the $t$-th copy of the set correspond to the singular vectors $\overline {\bf V}_{t,1}(\Omega)$ and $\widetilde {\bf U}_{t,1}(\Omega)$ residing in the $t$-th columns of  $\overline{\bf V}_i(\Omega)$ and $\widetilde {\bf U}_i(\Omega)$. 

More specifically, If  $\overline{\bf v}_i^{t}(\Omega)$ is the $t$-th column of $\overline{\bf V}_i(\Omega)$, then $\overline{\bf V}_{t,1}(\Omega):=[\overline{\bf v}_1^t(\Omega), \dots, \overline{\bf v}^t_{k_1}(\Omega)]$ and analogously with $\widetilde{\bf u}_i^t(\Omega)$,  $\widetilde {\bf U}_i(\Omega)$ and $\widetilde {\bf U}_{t,1}(\Omega)=[\widetilde{\bf u}_1^t(\Omega), \dots, \widetilde{\bf u}_{k_1}^t(\Omega)]$. The indexes $t,1$ indicate that among the singular values in the first $k_1$-multiplexed orbit of the permutation $\tau$,  we are investigating the $t$-th copy. Do the same for $\widetilde {\bf V}_{t,1}(\Omega)$ from $\widetilde {\bf V}_{i}(\Omega)$ and notice that 
\begin{align*}
    \overline {\bf V}_{t,1}(\Omega) &=
e^{j\Omega/(k_1\kappa_1)}
\widetilde {\bf V}_{t,1}(\Omega)
\diag( \omega_{k_1\kappa_1}^i )_{i=0,\dots,k_1-1}\\
&=
e^{j\Omega/(k_1\kappa_1)}
\widetilde {\bf V}_{t,1}(\Omega)
{\bf D}_{k_1} (-2\pi/\kappa_i),
\end{align*}
where ${\bf D}(\Omega)$ is defined in \eqref{eq:WN}. 

By hypothesis, the columns of $\overline {\bf V}_{t,1}(\Omega)$ and $\widetilde {\bf U}_{t,1}(\Omega)$ are $k_1$-multiplexed for any $t$, so $\widetilde {\bf U}_{t,1}(\Omega+2\pi)= \widetilde {\bf U}_{t,1}(\Omega){\bf P}_{k_1}$ and the same with $\overline {\bf V}_{t,1}(\Omega)$ where ${\bf P}_{k_1}$ is the permutation matrix in \eqref{eq:PN}. Focusing on $t=1$, and recalling that also ${\bf W}_{k_1}(\Omega)$ is $k_1$-multiplexed, we have
\begin{align*}
\widetilde {\bf U}_{t,1}(\Omega){\bf W}_{k_1}(\Omega)^H
    &= \widetilde {\bf U}_{t,1}(\Omega){\bf P}_{k_1}{\bf P}_{k_1}^H{\bf W}_{k_1}(\Omega)^H\\
    &= \widetilde {\bf U}_{t,1}(\Omega+2\pi) {\bf W}_{k_1}(\Omega + 2\pi)^H, 
\end{align*}
and the same applies with $\overline {\bf V}_{t,1}(\Omega)$ instead of $\widetilde{\textbf{U}}_{t,1}(\Omega)$. As a consequence, we can define  the matrices $\textbf{U}_{t,1}(\Omega):= \widetilde {\bf U}_{t,1}(\Omega) {\bf W}_{k_1}(\Omega)^H$ and $\textbf{V}_{t,1}(\Omega):= \overline{\textbf{V}}_{t,1}(\Omega){\bf W}_{k_1}(\Omega)^H$ that are both analytic and $2\pi$ periodic and with orthonormal columns. 
If now ${\bf \Sigma}_1(\Omega):=\text{diag}(\sigma_i(\Omega))_{i=1,\dots,k_1}$and ${\bf \Delta}_1(z):=\text{diag}(s_i(\Omega))_{i=1,\dots,k_1}$, then
\begin{align*}
\widetilde {\bf U}_{t,1}(\Omega)& {\bf \Delta}_1(\Omega) \overline {\bf V}_{t,1}(\Omega)^H\\
&= 
   \widetilde {\bf U}_{t,1}(\Omega) {\bf W}_{k_1}(\Omega)^H
   {\bf W}_{k_1}(\Omega){\bf \Delta}_1(\Omega){\bf W}_{k_1}(\Omega)^H
   {\bf W}_{k_1}(\Omega) \overline {\bf V}_{t,1}(\Omega)^H\\
  &= 
  {\bf U}_{t,1}(\Omega)
  {\bf C}_1(\Omega)
  {\bf V}_{t,1}(\Omega)^H,
\end{align*}
where  $\textbf{C}_1(\Omega):= {\bf W}_{k_1}(\Omega){\bf \Delta}_1(\Omega){\bf W}_{k_1}(\Omega)^H$ is a $2\pi$ periodic analytic pseudo-circulant matrix relative to the first set of singular values, since ${\bf \Delta}_1(\Omega)$ has the diagonal elements that are $k_1$-multiplexed. Repeating the same reasoning for all the copies of all the orbits of $\tau$, we find a decomposition of the form  $${\bf A}(z) = \sum_{t,i}{\textbf{U}}_{t,i}(z){\bf C}_i(z)
   {\textbf{V}}_{t,i}^P(z) = {\textbf{U}}(z){\bf C}(z)
   {\textbf{V}}^P(z)$$
 with pseudo-circulant  block-diagonal matrix ${\bf C}(z)$. 
\end{proof}

In a sense, the pseudo-circulant decomposition tells us that the only causes of non-holomorphicity in the SVD is given separately by each multiplexed set of singular values, and they can be removed simultaneously from the singular vectors and the singular values by allowing a block diagonal middle matrix with specially structured blocks. Moreover, the possibly signed singular values and vectors can be transformed nonetheless into a pseudo-circulant structure.

\begin{example}
Let ${\bf A}(z)$ be as in Example \ref{ex:2}. We know that ${\bf A}(\Omega) = {\bf U}(\Omega)
{\bf S}(\Omega) {\bf V}(\Omega)^H$ with 
\begin{align*}
    {\bf U}(\Omega) &:=-\frac 12
\begin{pmatrix}
    j(e^{-j2\Omega} +e^{j\Omega/2})   & j(e^{-j2\Omega}- e^{j\Omega/2})\\
  e^{-j2\Omega} - e^{j\Omega/2}   &e^{-j2\Omega} +e^{j\Omega/2}
\end{pmatrix}
,\\
{\bf S}(\Omega)& :=
\begin{pmatrix}
   2\cos(\Omega/4)e^{j\Omega/4} & 0& 0 \\
    0& -2j\sin(\Omega/4)e^{j\Omega/4} &0 
\end{pmatrix},
\\
{\bf V}(\Omega) &:= \frac 12
\begin{pmatrix}
    \sqrt 2 & \sqrt 2 & 0\\
   e^{-j\Omega/2} & -e^{-j\Omega/2} & \sqrt 2  \\
   e^{-j\Omega/2} & - e^{-j\Omega/2} & -\sqrt 2
\end{pmatrix},
\end{align*}
where all the $s_i(\Omega)$ and the columns ${\bf U}_1(\Omega)$, ${\bf U}_2(\Omega)$, ${\bf V}_1(\Omega)$, ${\bf V}_2(\Omega)$ are $2$-multiplexed. Since the diagonal entries of  ${\bf S}(\Omega)$ are $2$-multiplexed, the base change ${\bf W}_2(\Omega)$ in \eqref{eq:WN} transforms it into 
\[
{\bf C}(z):= {\bf W}_2(\Omega) {\bf S}(\Omega) \begin{pmatrix}
    {\bf W}_2(\Omega)^H &\\&1
\end{pmatrix}
=\begin{pmatrix}
    1&1&0\\z&1&0
\end{pmatrix},
\]
that is block pseudo-circulant with the only pseudo-circulant block relative to the functions $\phi_0(z)=1$, $\phi_1(z)=z$. As a consequence, one can find the holomorphic pseudo-circulant decomposition as ${\bf A}(z) = {\bf U}'(z){\bf C}(z){\bf V}'(z)^H$ with
\begin{align*}
    {\bf U}'(z) &:=  {\bf U}(z)  {\bf W}_2(z)^P  = \frac 1{\sqrt 2}
\begin{pmatrix}
    -jz^{-2}&-j\\-z^{-2}&1
\end{pmatrix}
,\\
{\bf V}'(z) &:= {\bf V}(z) \begin{pmatrix}
    {\bf W}_2(z)^P &\\&1
\end{pmatrix} = \frac 1{\sqrt 2}
\begin{pmatrix}
   \sqrt 2& &\\
   & z^{-1}&1\\
   &z^{-1}&-1 
\end{pmatrix}.
\end{align*}
\hfill \flushright{$\triangle$}
\end{example}

{
}

\section{Conclusions}

We have studied three different decompositions for possibly rectangular matrices ${\bf A}(z)$ that are analytic at least on $S^1$. From the known SVD that is analytic with respect to $z^{1/L}$, but with possibly negative singular values, we have shown how to reduce the parameter $L$ separately for every singular value, first by halving it by removing the sign change but relaxing the singular values in the complex plane, and then by transforming the multiplexed singular subspaces into an holomorphic pseudo-circulant structure. 
This transformation shows that sign ambiguity plays no role in a possible decoupling of different multiplexed systems, no matter how they have been interlaced, akin to the eigenvalue decomposition of para-Hermitian matrices.

Future work will focus on the design of algorithms for the computation of such decompositions, in particular how to retrieve the pseudo-circulant factorization without having to compute beforehand the singular values. 
Another feature to be yet studied is the stability of all these factorizations, whose analysis will be based on the preliminary results already shown in \cite{BV22} about the stability of the eigenvalues of para-Hermitian matrices. 

\section*{\textit{Acknowledgements}}

The author thanks professor Stephan Weiss for its help and insight on the topic and the presentation of the paper. 

This work was partially supported by Alfred Kordelinin s\"a\"ati\"o Grant No. 210122
and by the European Union ERC consolidator grant, eLinoR, no 101085607.

\end{document}